\documentclass[amstex,12pt,reqno]{amsart}
\usepackage{amsmath, amssymb, amsthm}
\usepackage{mathrsfs}

\usepackage{mathtools}
\mathtoolsset{showonlyrefs=true}

\newtheorem{theorem}{Theorem}[section]
\newtheorem{lemma}[theorem]{Lemma}
\newtheorem{proposition}[theorem]{Proposition}
\newtheorem{corollary}[theorem]{Corollary}

\theoremstyle{definition}

\newtheorem{remark}{Remark}

\title{A fiber bundle over BMO Teich\-m\"ul\-ler space}
\author[K. Matsuzaki]{Katsuhiko Matsuzaki}
\address{Department of Mathematics, School of Education, Waseda University \endgraf
Shinjuku, Tokyo 169-8050, Japan}
\email{matsuzak@waseda.jp}

\subjclass[2020]{Primary 30C62, 30F60, 30H35; Secondary 32G15, 46G20, 57R22}
\keywords{Teich\-m\"ul\-ler space, BMOA, VMOA, holomorphic split submersion, pre-Schwarzian derivative, 
real-analytic disk bundle, Bers fiber space}

\thanks{Research supported by 
Japan Society for the Promotion of Science (KAKENHI 23K25775 and 23K17656)}
 
\begin{document}

\maketitle

\begin{abstract}
We prove that the fiber space consisting of BMOA functions that are
the logarithms of derivatives of conformal homeomorphisms of the unit disk 
onto bounded quasidisks forms 
a real-analytic disk bundle over the Bers embedding of the BMO Teich\-m\"ul\-ler space. 
For the VMO Teich\-m\"ul\-ler space, we show that the corresponding sub-bundle consisting of VMOA functions
is real-analytically trivial.
\end{abstract}

\section{Introduction}

In contrast to the universal Teich\-m\"ul\-ler space, which can be regarded as the space of images of the unit circle under quasiconformal self-homeomorphisms of the complex plane, Astala and Zinsmeister \cite{AZ} introduced the BMO Teich\-m\"ul\-ler space by considering bi-Lipschitz self-homeomorphisms. This involves BMO functions, which play a central role in real and complex analysis (see \cite{Ga, Zhu}). Their framework was further developed by Shen and Wei \cite{SW} into a complex-analytic theory of Teich\-m\"ul\-ler spaces. On the other hand, using VMO functions characterized as a vanishing subclass of BMO functions under a degeneracy condition on the BMO norm, one can define the corresponding VMO Teich\-m\"ul\-ler space as a closed subspace of the BMO Teich\-m\"ul\-ler space.

A Teich\-m\"ul\-ler space can be identified with a deformation space of complex projective structures, and by taking the Schwarzian derivative of conformal maps, it is realized as a domain in a Banach space of holomorphic functions. This is known as the Bers embedding, which equips the Teich\-m\"ul\-ler space with a complex structure. In cases such as the universal Teich\-m\"ul\-ler space, where the action of a Fuchsian group is not considered, one may instead work with the pre-Schwarzian derivative
yielding a similar realization with certain distinctions.

In the case of the universal Teich\-m\"ul\-ler space, Zhuravlev \cite{Z} proved that within the Banach space of Bloch functions on the unit disk, the set of logarithms of derivatives of conformal homeomorphisms of the unit disk onto quasidisks (bounded or unbounded)
% that extend quasiconformally to the Riemann sphere 
forms an open set with uncountably many connected components. The component containing the origin (bounded case) corresponds to the fiber space over the Teich\-m\"ul\-ler space, which is biholomorphically equivalent to the Bers fiber space. In contrast, all other components (unbounded case) are biholomorphic to the Teich\-m\"ul\-ler space. A similar phenomenon occurs in the Banach space of BMOA functions on the unit disk in the setting of the BMO Teich\-m\"ul\-ler space, as noted in \cite{AZ} and further analyzed in detail in \cite{SW}.

In this paper, we study the structure of the projection from the space of BMOA functions that are the logarithms of derivatives of conformal maps from the unit disk onto bounded quasidisks, which arises as the fiber space over the BMO Teich\-m\"ul\-ler space. We show that this space forms a real-analytic disk bundle (Theorem~\ref{bundle2}). A corresponding result holds for the VMO Teich\-m\"ul\-ler space, where the bundle structure restricts to the sub-bundle consisting of VMOA functions; in this case, we further show that the real-analytic disk sub-bundle is trivial (Theorem~\ref{product}). 

As a side remark, for the VMO Teich\-m\"ul\-ler space, the phenomenon discovered by Zhuravlev does not occur: we prove that in the Banach space of VMOA functions, the set of logarithms of derivatives of conformal maps that extend quasiconformally consists of a single connected open set containing the origin (Theorem~\ref{onlyL}).

\medskip
\noindent
{\bf Acknowledgements.}
The author thanks the referee for his/her careful reading of the manuscript and giving valuable comments, which have helped to improve the clarity of this work; Proposition \ref{affine} is corrected and Remark \ref{explanation} is added.

\section{BMO and VMO Teich\-m\"ul\-ler spaces}

In this preliminary section, we introduce the BMO and VMO Teich\-m\"ul\-ler spaces in comparison with  
the universal Teich\-m\"ul\-ler space.  
We assume familiarity with the basic concepts of the theory of the universal Teich\-m\"ul\-ler space and do not explain them here.  
The reader may consult \cite{Le} for further details.

Let $\mathbb D$ be the unit disk in the complex plane $\mathbb C$, and let $\mathbb D^*=\widehat{\mathbb C} \setminus \overline{\mathbb D}$,  
where $\widehat{\mathbb C}=\mathbb C \cup \{\infty\}$ denotes the Riemann sphere.  
A Borel measure $m$ on $\mathbb D$ (or on $\mathbb D^*$) is called a {\it Carleson measure} if
$$
\Vert m \Vert_c=\sup_{x \in \mathbb S,\, r>0} \frac{m(\Delta(x,r) \cap \mathbb D)}{r}<\infty,
$$
where $\Delta(x,r)$ denotes the disk of radius $r>0$ centered at $x$ on the unit circle 
$\mathbb S=\partial \mathbb D$. A Carleson measure $m$ is said to be {\it vanishing} if
$$
\lim_{r \to 0} \frac{m(\Delta(x,r) \cap \mathbb D)}{r}=0
$$
uniformly in $x$. The set of Carleson measures on $\mathbb D$ is denoted by ${\rm CM}(\mathbb D)$,  
and the set of vanishing Carleson measures by ${\rm CM}_0(\mathbb D)$.  
Carleson measures on $\mathbb D^*$ are defined analogously.

Let $M(\mathbb D^*)=\{\mu \in L^\infty(\mathbb D^*) \mid \Vert \mu \Vert_\infty<1\}$ be the space of Beltrami coefficients on $\mathbb D^*$.  
We define the following subspaces of $M(\mathbb D^*)$, 
whose elements induce absolutely continuous Carleson measures:
\begin{align*}
M_B(\mathbb D^*)&=\{\mu \in M(\mathbb D^*) \mid m_\mu \in {\rm CM}(\mathbb D^*)\}, \ dm_\mu:=\frac{|\mu(z)|^2}{|z|^2-1}dxdy;\\
M_V(\mathbb D^*)&=\{\mu \in M(\mathbb D^*) \mid m_\mu \in {\rm CM}_0(\mathbb D^*)\}.
\end{align*}
The norm on $M_B(\mathbb D^*)$ is defined by $\Vert \mu \Vert_*=\Vert \mu \Vert_\infty +\Vert m_\mu \Vert_c^{1/2}$.  
Then, $M_V(\mathbb D^*)$ is a closed subspace of $M_B(\mathbb D^*)$.

The universal Teich\-m\"ul\-ler space $T$ is the quotient of $M(\mathbb D^*)$ by Teich\-m\"ul\-ler equivalence.  
Here, $\mu$ and $\nu$ in $M(\mathbb D^*)$ are said to be equivalent if the normalized quasiconformal self-homeomorphisms  
$H(\mu)$ and $H(\nu)$ of $\mathbb D^*$, having complex dilatations $\mu$ and $\nu$, respectively,  
extend to the same quasisymmetric homeomorphism of $\mathbb S$.  
The quotient map $\pi:M(\mathbb D^*) \to T$ is called the {\it Teich\-m\"ul\-ler projection}. 
The quotient topology is equipped from $M(\mathbb D^*)$ to $T$.
In this setting, the BMO Teich\-m\"ul\-ler space $T_B$ is defined as $\pi(M_B(\mathbb D^*))$,  
and the VMO Teich\-m\"ul\-ler space $T_V$ as $\pi(M_V(\mathbb D^*))$, which is closed in $T_B$
under the quotient topology induced from $M_B(\mathbb D^*)$. 

For $\mu \in M(\mathbb D^*)$, let $F_\mu$ denote the conformal homeomorphism of $\mathbb D$ onto a bounded domain (quasidisk) in $\mathbb C$  
with $F_\mu(0)=0$ and $(F_\mu)'(0)=1$, which extends to a quasiconformal self-homeomorphism of $\mathbb C$ with complex dilatation $\mu$.  
We assume that $F_\mu$ is extended so that $F_\mu(\infty)=\infty$. 
This map is uniquely determined by $\mu$.   
Then, the Schwarzian derivative $\Psi=S_{F_\mu}$ of $F_\mu$ belongs to the Banach space $A(\mathbb D)$ of  
holomorphic functions $\Psi$ on $\mathbb D$ for which the norm $\Vert \Psi \Vert_A=\sup_{z \in \mathbb D}(1-|z|^2)^2|\Psi(z)|$ is finite.  
We call this correspondence $S:M(\mathbb D^*) \to A(\mathbb D)$, given by $\mu \mapsto S_{F_\mu}$,  
the {\it Schwarzian derivative map}.  
It is known that $S$ is a holomorphic split submersion onto its image (see \cite[Section 3.4]{N}).

\begin{remark}\label{explanation}
We clarify the definition of a split submersion, since this concept will play an important role in our later arguments.
In general, a differentiable surjection $f:X \to Y$ between Banach manifolds $X$ and $Y$ is called a
split submersion if, at every point $x \in X$, the derivative $d_xf:T_xX \to T_{f(x)}Y$ between the tangent spaces
is surjective and admits a bounded linear right inverse $\tau:T_{f(x)}Y \to T_xX$. By the implicit function theorem, this is equivalent to requiring that for every $x \in X$
there exists a local differentiable right inverse $\sigma$ of $f$,
defined on some neighborhood of $f(x) \in Y$, such that $\sigma(f(x))=x$
(see \cite[p.89]{N}). We emphasize that in this equivalent statement 
the condition $\sigma(f(x))=x$ is necessary. 
If this condition is omitted, we shall refer to $f$ as a split submersion in the weak sense.
\end{remark}

%\begin{remark}\label{explanation}
%We clarify the definition of split submersion because this concept is crucial in our arguments later.
%In general, a differentiable surjection $f:X \to Y$ between Banach manifolds $X$ and $Y$ is called a
%split submersion if at every point $x \in X$, the derivative $d_xf:T_xX \to T_{f(x)}Y$ between the tangent spaces
%is surjective and there exists a bounded linear right inverse $\tau:T_{f(x)}Y \to T_xX $ of $d_xf$. By the implicit function theorem, this is equivalent to saying that for every $x \in X$
%there exists a local differentiable right inverse $\sigma$ of $f$
%defined on some neighborhood of $f(x) \in Y$ with $\sigma(f(x))=x$.
%See \cite[p.89]{N}. We note that in this equivalent statement 
%$\sigma$ is required to
%satisfy $\sigma(f(x))=x$. 
%If this condition is dropped, we say that $f$ is a split submersion in the weak sense.
%\end{remark}

The corresponding Banach spaces of holomorphic functions on $\mathbb D$ associated with 
$M_B(\mathbb D^*)$ and $M_V(\mathbb D^*)$  
are defined as follows:
\begin{align*}
A_B(\mathbb D)&=\{\Psi \in A(\mathbb D) \mid m_\Psi \in {\rm CM}(\mathbb D)\}, \ dm_\Psi:=(1-|z|^2)^3|\Psi(z)|^2dxdy;\\
A_V(\mathbb D)&=\{\Psi \in A(\mathbb D) \mid m_\Psi \in {\rm CM}_0(\mathbb D)\}.
\end{align*}
The norm on $A_B(\mathbb D)$ is defined by $\Vert \Psi \Vert_{A_B} =\Vert m_\Psi \Vert_c^{1/2}$.  
Then, $A_V(\mathbb D)$ is a closed subspace of $A_B(\mathbb D)$.

The images of $M_B(\mathbb D^*)$ and $M_V(\mathbb D^*)$ under the Schwarzian derivative map $S$ are  
contained in $A_B(\mathbb D)$ and $A_V(\mathbb D)$, respectively, and moreover, they satisfy
$$
S(M_B(\mathbb D^*))=S(M(\mathbb D^*)) \cap A_B(\mathbb D),\quad 
S(M_V(\mathbb D^*))=S(M(\mathbb D^*)) \cap A_V(\mathbb D).
$$
In addition with these, the following result is also proved by Shen and Wei \cite[Theorem 5.1]{SW}.

\begin{proposition}\label{weaksubmersion}
The map $S:M_B(\mathbb D^*) \to A_B(\mathbb D)$ is holomorphic and admits a local holomorphic right inverse at every point  
in the image $S(M_B(\mathbb D^*))$. The restriction of $S$ to $M_V(\mathbb D^*)$ enjoys the analogous property  
as a holomorphic map into $A_V(\mathbb D)$.
\end{proposition}

The Schwarzian derivative map $S:M(\mathbb D^*) \to A(\mathbb D)$ factors through the Teich\-m\"ul\-ler projection  
$\pi:M(\mathbb D^*) \to T$ into an injection $\alpha:T \to A(\mathbb D)$, which is called the {\it Bers embedding}. 
The map $\alpha$ is continuous since $S$ is continuous. Its image  
$\alpha(T)=S(M(\mathbb D^*))$ is open in $A(\mathbb D)$ and $\alpha^{-1}$ is continuous on $\alpha(T)$ since
there is a local continuous right inverse of $S$
due to the split submersion property of $S$. Hence, $\alpha$ is a homeomorphism onto   
$\alpha(T)$, endowing $T$ with a complex Banach manifold structure.  

Analogously, the Bers embedding for $T_B$ and $T_V$ is given by restricting $\alpha$ to these spaces.  
Proposition \ref{weaksubmersion} ensures that $\alpha:T_B \to A_B(\mathbb D)$  
and $\alpha:T_V \to A_V(\mathbb D)$ are homeomorphisms onto their respective images  
$\alpha(T_B)=S(M_B(\mathbb D^*))$ and $\alpha(T_V)=S(M_V(\mathbb D^*))$.  
Thus, $T_B$ and $T_V$ inherit Banach manifold structures, with $T_V$ being a closed submanifold of $T_B$.  
We remark that in \cite{SW}, the term ``split submersion" is used in the weaker sense, referring only to  
the property stated in Proposition \ref{weaksubmersion}. However, this does not affect the fact that $\alpha$
is a homeomorphism.

\section{The components in BMOA and VMOA}

In this section, we consider the map induced by the pre-Schwarzian derivative, 
instead of the Schwarzian derivative discussed in the previous section.
For $\mu \in M(\mathbb D^*)$, the function $\Phi=\log(F_\mu)'$ belongs to the space $B(\mathbb D)$ of  
{\it Bloch functions} on $\mathbb D$, with seminorm $\Vert \Phi \Vert_B=\sup_{z \in \mathbb D}(1-|z|^2)|\Phi'(z)|$ finite.  
Here, $\Phi'$ is the pre-Schwarzian derivative $N_{F_\mu}=(F_\mu)''/(F_\mu)'$ of $F_\mu$.  
By identifying Bloch functions that differ by a constant, we regard  
$B(\mathbb D)$ as a Banach space with norm $\Vert \Phi \Vert_B$.  
We call the correspondence $L:M(\mathbb D^*) \to B(\mathbb D)$, given by $\mu \mapsto \log(F_\mu)'$,  
the {\it pre-Schwarzian derivative map}. 
Strictly speaking, the pre-Schwarzian derivative $N_{F_\mu}$ is the derivative of $\Phi=\log(F_\mu)'$, 
so the terminology may be somewhat misleading. However, since the norm on $B(\mathbb D)$ is defined in terms of $\Phi'$, 
we adopt this name. Alternatively, one might think of ``pre'' here as indicating a ``pre$^2$'' version.  
As in the case of $S$, the map $L$ is holomorphic.

%In this section, we consider the map induced by the pre-Schwarzian derivative, 
%instead of the Schwarzian derivative discussed in the previous section.
%For $\mu \in M(\mathbb D^*)$, the function $\Phi=\log(F_\mu)'$ belongs to the space $B(\mathbb D)$ of  
%{\it Bloch functions} on $\mathbb D$, with seminorm $\Vert \Phi \Vert_B=\sup_{z \in \mathbb D}(1-|z|^2)|\Phi'(z)|$ finite.  
%Here, $\Phi'$ is the pre-Schwarzian derivative $N_{F_\mu}=(F_\mu)''/(F_\mu)'$ of $F_\mu$.  
%By identifying Bloch functions differing by a constant, we regard  
%$B(\mathbb D)$ as a Banach space with norm $\Vert \Phi \Vert_B$.  
%We call the correspondence $L:M(\mathbb D^*) \to B(\mathbb D)$, given by $\mu \mapsto \log(F_\mu)'$,  
%the {\it pre-Schwarzian derivative map}. 
%Actually, the pre-Schwarzian derivative $N_{F_\mu}$ is the derivative of $\Phi=\log(F_\mu)'$ and
%this name may be confusing, but since the norm of $B(\mathbb D)$ is given by taking the derivative of $\Phi$,
%we call $L$ in this way. Or we may regard ``pre'' as pre$^2$.
%Then, $L$ is holomorphic as in the case of $S$.

To study the relationship between $B(\mathbb D)$ and $A(\mathbb D)$,  
we define a map $J:B(\mathbb D) \to A(\mathbb D)$ by  
$\Phi \mapsto \Psi=\Phi''-(\Phi')^2/2$, according to the definition of the Schwarzian derivative.  
It is known that $J$ is holomorphic and satisfies $S=J \circ L$ (see \cite[Corollary A.3]{Teo}).  
Therefore, $L$ is also a holomorphic split submersion onto its image $\widetilde{\mathcal T}=L(M(\mathbb D^*))$ 
in $B(\mathbb D)$.  
By restricting $J$ to $\widetilde{\mathcal T}$, we obtain a holomorphic split submersion  
$J:\widetilde{\mathcal T} \to \alpha(T)$.  
The non-injectivity of $J$ and the elements of its inverse image are seen from the following proposition.

\begin{proposition}\label{affine}
{\rm (i)} 
For $\mu, \nu \in M(\mathbb D^*)$, we have $S_{F_{\mu}}=S_{F_{\nu}}$ if and only if  
$F_{\mu}=W \circ F_{\nu}$ on $\mathbb D$
for some M\"obius transformation $W$ of $\widehat{\mathbb C}$ such that  
$W \circ F_{\nu}(\mathbb D)$ is a bounded domain in $\mathbb C$. 
Moreover, $N_{F_{\mu}}=N_{F_{\nu}}$ if and only if
$F_{\mu}=W \circ F_{\nu}$ on $\mathbb D$
for some affine transformation $W$ of $\mathbb C$.
{\rm (ii)}
For any $\nu \in M(\mathbb D^*)$ and any M\"obius transformation $W$ such that $W \circ F_{\nu}(\mathbb D)$ is a bounded domain,  
there exists some $\mu' \in M(\mathbb D^*)$ such that $N_{W \circ F_{\nu}}=N_{F_{\mu'}}$. 
%$W \circ F_{\nu}$ maps $\mathbb D$ onto a bounded domain in $\mathbb C$. 
%Moreover, $W$ is an affine transformation of $\mathbb C$ if and only if  
%$N_{W \circ F_{\nu}}=N_{F_{\nu}}$.
\end{proposition}

\begin{proof}
(i) By the invariance of the Schwarzian derivative $S$, we have $S_{F_{\mu}}=S_{F_{\nu}}$ if and only if
$F_{\mu}=W \circ F_{\nu}$ on $\mathbb D$ for some M\"obius transformations $W$ of $\widehat{\mathbb C}$.
However, since $\infty \notin \overline{F_{\mu}(\mathbb D)}$ by the normalization of $F_{\mu}$, $\infty$ is not contained in
$\overline{W \circ F_{\nu}(\mathbb D)}$.
This shows the first statement.
The second statement follows from the invariance of the pre-Schwarzian derivative $N$  
under post-composition with affine transformations of $\mathbb C$.

(ii) We note that $W \circ F_{\nu}(\infty)=w$ lies in  
$\widehat{\mathbb C} \setminus \overline{W \circ F_{\nu}(\mathbb D)}$.  
Then, one can construct a quasiconformal self-homeomorphism $G$ of $\widehat{\mathbb C}$  
with $G(w)=\infty$ that is the identity on a neighborhood of $\overline{W \circ F_{\nu}(\mathbb D)}$.  
Letting $\mu'$ be the complex dilatation of $G \circ W \circ F_{\nu}$ on $\mathbb D^*$,  
we define $F_{\mu'}=G \circ W \circ F_{\nu}$. This satisfies $N_{F_{\mu'}}=N_{W \circ F_{\nu}}$.
\end{proof}

Proposition \ref{affine} also applies to $M_B(\mathbb D^*)$ and $M_V(\mathbb D^*)$.  
To adapt statement (ii) to these cases, it suffices to observe the construction of  
$G$ and $\mu'$. Since $G$ is the identity on a neighborhood of $\overline{W \circ F_{\nu}(\mathbb D)}$,  
we have $\mu'(\zeta)=\nu(\zeta)$ on $1<|\zeta|<R$ for some $R>1$.  
Hence, if $\nu \in M_B(\mathbb D^*)$, then $\mu' \in M_B(\mathbb D^*)$;  
similarly for $M_V(\mathbb D^*)$.

Next, we introduce a different normalization for $F_\mu$.  
Namely, for any $\xi \in \mathbb S$, let $F_\mu^\xi$ denote the conformal map on $\mathbb D$  
that extends to a quasiconformal self-homeomorphism of $\widehat{\mathbb C}$ satisfying  
$F_\mu^\xi(0)=0$, $(F_\mu^\xi)'(0)=1$, and $F_\mu^\xi(\xi)=\infty$.  
Then, $\log (F_\mu^\xi)'$ also belongs to $B(\mathbb D)$, and  
$L^\xi:M(\mathbb D^*) \to B(\mathbb D)$, defined by $\mu \mapsto \log(F^\xi_\mu)'$,  
is a holomorphic split submersion onto its image $\mathcal T^\xi=L^\xi(M(\mathbb D^*))$.  
We remark that if we choose $\zeta \in \mathbb D^*$ for the normalization $\zeta \mapsto \infty$, the map $L^\zeta$  
satisfies $L^\zeta(M(\mathbb D^*))=L^\infty(M(\mathbb D^*))=\widetilde{\mathcal T}$. 
This can be also seen from Proposition \ref{affine} (ii). 

Zhuravlev \cite{Z} proved the following.

\begin{proposition}\label{Z}
The family $\{\mathcal T^\xi\}_{\xi \in \mathbb S}$ consists of mutually disjoint connected open subsets of  
$B(\mathbb D)$, each biholomorphically equivalent to $T$.  
They are all disjoint from $\widetilde{\mathcal T}$.  
Moreover, for a conformal homeomorphism $F$ of $\mathbb D$ that extends quasiconformally 
to $\widehat{\mathbb C}$,  
if $\Phi=\log F' \in B(\mathbb D)$, then  
$\Phi$ belongs to $\bigsqcup_{\xi \in \mathbb S} \mathcal T^\xi \sqcup \widetilde{\mathcal T}$.
\end{proposition}

We now study the pre-Schwarzian derivative map on $M_B(\mathbb D^*)$ and $M_V(\mathbb D^*)$.  
The corresponding Banach spaces consist of {\it BMOA and VMOA functions} on $\mathbb D$:
\begin{align*}
{\rm BMOA}(\mathbb D)&=\{\Phi \in B(\mathbb D) \mid m_\Phi \in {\rm CM}(\mathbb D)\},
\ dm_\Phi:=(1-|z|^2)|\Phi'(z)|^2dxdy;\\
{\rm VMOA}(\mathbb D)&=\{\Phi \in B(\mathbb D) \mid m_\Phi \in {\rm CM}_0(\mathbb D)\}.
\end{align*}
The norm on ${\rm BMOA}(\mathbb D)$ is defined by $\Vert \Phi \Vert_{\rm BMOA} =\Vert m_\Phi \Vert_c^{1/2}$.  
Then, ${\rm VMOA}(\mathbb D)$ is a closed subspace of ${\rm BMOA}(\mathbb D)$.  
Note that there are several equivalent conditions for BMOA (see \cite{Ga, Pom, Zhu}),  
but the above definition is the one adopted in \cite{AZ}. For their equivalence, see \cite[Theorem 6.5]{Gi}.  

The images of $M_B(\mathbb D^*)$ and $M_V(\mathbb D^*)$ under the pre-Schwarzian derivative map $L$,  
denoted by $\widetilde{\mathcal T}_B$ and $\widetilde{\mathcal T}_V$, are contained in  
${\rm BMOA}(\mathbb D)$ and ${\rm VMOA}(\mathbb D)$ respectively. Moreover, we have  
$\widetilde{\mathcal T}_B=\widetilde{\mathcal T} \cap {\rm BMOA}(\mathbb D)$ and  
$\widetilde{\mathcal T}_V=\widetilde{\mathcal T} \cap {\rm VMOA}(\mathbb D)$ (see \cite[p.143]{SW}).

We consider the map $J$ on ${\rm BMOA}(\mathbb D)$ and ${\rm VMOA}(\mathbb D)$.  
Then, their images lie in $A_B(\mathbb D)$ and $A_V(\mathbb D)$, respectively, and  
$J:{\rm BMOA}(\mathbb D) \to A_B(\mathbb D)$ is holomorphic.  
Since $J \circ L = S$, we obtain a holomorphic surjection  
$J: \widetilde{\mathcal T}_B \to \alpha(T_B)$ (where we use the same notation), with  
$J(\widetilde{\mathcal T}_V) = \alpha(T_V)$.  
The fact that $J$ is not injective on $\widetilde{\mathcal T}_B$ and $\widetilde{\mathcal T}_V$ also follows from  
Proposition \ref{affine}.

We now examine the other connected components in ${\rm BMOA}(\mathbb D)$, beyond $\widetilde{\mathcal T}_B$,  
that arise from $\log F'$ for quasiconformally extendable conformal homeomorphisms $F$ of $\mathbb D$.  
More precisely, they are given by the holomorphic map $L^\xi$ on $M_B(\mathbb D^*)$ 
for each $\xi \in \mathbb S$. Let $\mathcal T^\xi_B = L^\xi(M_B(\mathbb D^*))$,
which is contained in ${\rm BMOA}(\mathbb D)$.  
As noted in \cite{AZ}, Zhuravlev's theorem (Proposition \ref{Z}) holds true when  
$\mathcal T^\xi$ is replaced by $\mathcal T_B^\xi$. In fact, 
$\mathcal T^\xi_B = \mathcal T^\xi \cap {\rm BMOA}(\mathbb D)$ for all $\xi \in \mathbb S$,
and they are biholomorphically equivalent to $T_B$ under $J$
by \cite[Theorem 6.2]{SW}.  
%These statements ultimately stem from the fact that  
%$\log(z - \xi)$ belongs to ${\rm BMOA}(\mathbb D)$.

%\begin{proposition}
%$\mathcal T^\xi_B = \mathcal T^\xi \cap {\rm BMOA}(\mathbb D)$ for any $\xi \in \mathbb S$.
%\end{proposition}

In contrast, in the case of the VMO Teich\-m\"ul\-ler space $T_V$, with $M_V(\mathbb D^*)$ and ${\rm VMOA}(\mathbb D)$,  
Theorem \ref{onlyL} below shows that  
$L^\xi(M_V(\mathbb D^*)) \cap {\rm VMOA}(\mathbb D) =\emptyset$ for every $\xi \in \mathbb S$.  
This result is a consequence of a more general statement concerning the space $M_0(\mathbb D^*)$ of  
{\it asymptotically conformal} Beltrami coefficients, and the little Bloch space $B_0(\mathbb D)$, defined as follows:
\begin{align*}
M_0(\mathbb D^*) &= \{ \mu \in M(\mathbb D^*) \mid \lim_{t \to 0} \underset{|z|<1+t}{\mathrm{ess}\sup} |\mu(z)| = 0 \};\\
B_0(\mathbb D) &= \{ \Phi \in B(\mathbb D) \mid \lim_{t \to 0} \sup_{|z|>1-t} (1 - |z|^2) |\Phi'(z)| = 0 \}.
\end{align*}
It is known that both are closed subspaces of their respective ambient spaces, and that  
$L(M_0(\mathbb D^*)) \subset B_0(\mathbb D)$.

\begin{lemma}\label{noother}
Let $\nu \in M_0(\mathbb D^*)$. If $\Phi = \log(W \circ F_\nu)'$ belongs to  
$B_0(\mathbb D)$ for some M\"obius transformation $W$ of $\widehat{\mathbb C}$,  
then $W \circ F_\nu$ maps $\mathbb D$ onto a bounded domain in $\mathbb C$.
\end{lemma}

\begin{proof}
Let $a = W^{-1}(\infty)$. Since $W \circ F_\nu$ is holomorphic on $\mathbb D$,  
it follows that $a \notin F_\nu(\mathbb D)$, so $a \in \widehat{\mathbb C} \setminus F_\nu(\mathbb D)$.  
We aim to exclude the case where $a \in \partial F_\nu(\mathbb D)$.  
If we can do so, then it must be that $a \notin \overline{F_\nu(\mathbb D)}$,  
implying that $W \circ F_\nu(\mathbb D)$ is bounded.

By straightforward computation, we obtain:
\begin{equation}\label{chain}
\Phi'(z) = N_{W \circ F_\nu}(z) = N_W(F_\nu(z)) \cdot F_\nu'(z) + N_{F_\nu}(z) 
= \frac{-2F_\nu'(z)}{F_\nu(z) - a} + N_{F_\nu}(z).
\end{equation}
Hence, the assumptions $\nu \in M_0(\mathbb D^*)$ and $\Phi \in B_0(\mathbb D)$ imply that
\begin{equation}\label{littlecondition}
\lim_{|z| \to 1} (1 - |z|^2) \cdot \frac{|F_\nu'(z)|}{|F_\nu(z) - a|} = 0.
\end{equation}
Let $\delta(w)$ denote the reciprocal of the hyperbolic density in $F_\nu(\mathbb D)$,  
so that $\delta(F_\nu(z)) = (1-|z|^2) |F_\nu'(z)|$.  
It is well known that $\delta(w)$ is comparable to the Euclidean distance  
$d(w, \partial F_\nu(\mathbb D))$ from $w$ to the boundary.  
Then, from \eqref{littlecondition} we obtain:
\begin{equation}\label{littlecondition2}
\lim_{|z| \to 1} \frac{d(F_\nu(z), \partial F_\nu(\mathbb D))}{|F_\nu(z) - a|} = 0.
\end{equation}

Suppose $a \in \partial F_\nu(\mathbb D)$, and let $z_0 = F_\nu^{-1}(a) \in \mathbb S$.  
For $z = r z_0 \in \mathbb D$ with $0 < r < 1$, the limit in \eqref{littlecondition2} implies:
\begin{equation}\label{John}
\lim_{r \to 1} \frac{d(F_\nu(r z_0), \partial F_\nu(\mathbb D))}{|F_\nu(r z_0) - F_\nu(z_0)|} = 0.
\end{equation}
However, the fraction in \eqref{John} is bounded away from zero,  
because $F_\nu(\mathbb D)$ is a John domain  
(see \cite[Corollary 5.3]{Pom}).  
This contradiction proves that $a \notin \partial F_\nu(\mathbb D)$, as required.
\end{proof}

Lemma \ref{noother} implies that there are no components $L^\xi(M_0(\mathbb D^*))$ 
with $\xi \in \mathbb S$ in $B_0(\mathbb D)$.  
Since ${\rm VMOA}(\mathbb D) \subset B_0(\mathbb D)$ (see \cite[Corollary 5.2]{Gi}, \cite[p.280]{Zhu}),
we obtain the following result.

\begin{theorem}\label{onlyL}
%The set of all holomorphic functions $\Phi = \log (W \circ F_\mu)'$ in ${\rm VMOA}(\mathbb D)$,  
%given by M\"obius transformations $W$ of $\widehat{\mathbb C}$ and  
%by $\mu \in M_V(\mathbb D^*)$, coincides with $\widetilde{\mathcal T}_V$.
The set of all holomorphic functions $\Phi = \log F'$ in ${\rm VMOA}(\mathbb D)$,
given by conformal homeomorphisms $F$ of $\mathbb D$ that extend quasiconformally 
to $\widehat{\mathbb C}$ with their complex dilatations in $M_V(\mathbb D^*)$,  
coincides with $\widetilde{\mathcal T}_V$.
\end{theorem}

\section{The structure of the fiber space}

We investigate the structure of the holomorphic surjection $J:\widetilde{\mathcal T}_B \to \alpha(T_B)$
in greater detail.  
From the arguments in the previous section, we observe that the multivalency of $J$ arises from the post-composition
$W_a \circ F_\nu$ of the quasiconformally extendable conformal homeomorphism 
$F_\nu$ for $\nu \in M_B(\mathbb D^*)$
with a M\"obius transformation $W_a$ sending $a \in F_\nu(\mathbb D^*)$ to $\infty$. 
Since $W_a \circ F_\nu$ is uniquely determined by $\nu$ and $a$, 
up to post-composition with affine transformations of $\mathbb C$,
we consider the map 
$\widetilde L(\nu,a)=\log(W_a \circ F_\nu)' \in \widetilde{\mathcal T}_B$. This is 
defined on the fiber space over $M_B(\mathbb D^*)$ given by
\[
\widetilde M_B(\mathbb D^*)=\{(\nu,a) \in M_B(\mathbb D^*) \times \widehat{\mathbb C} 
\mid a \in F_{\nu}(\mathbb D^*)\},
\]
which is a domain in the product manifold $M_B(\mathbb D^*) \times \widehat{\mathbb C}$. 
Note that $\widetilde L(\nu,\infty)=L(\nu)$.

\begin{lemma}\label{localbound}
$\widetilde L:\widetilde M_B(\mathbb D^*) \to \widetilde{\mathcal T}_B$ is holomorphic.
\end{lemma}

\begin{proof}
Let $\Phi_0=\widetilde L(\nu,\infty)=\log (F_\nu)'$ and $\Phi=\widetilde L(\nu,a)=\log (W_a \circ F_\nu)'$.
By \eqref{chain}, we have
\begin{align}\label{1st2nd}
\Phi'(z)
%=N_{W_a \circ F_\nu}(z)=N_{W_a} \circ F_\nu(z)\cdot (F_\nu)'(z)+N_{F_\nu}(z)
=\frac{-2(F_\nu)'(z)}{{F_\nu}(z)-a}+\Phi_0'(z).
\end{align}
When $a=\infty$, this simply reads as $\Phi'(z)=\Phi_0'(z)$; hence, we may assume $a \neq \infty$.
Since $a \in F_\nu(\mathbb D^*)$, the denominator $F_\nu(z)-a$ for $z \in \mathbb D$ is 
bounded below by the distance $d(a,\partial F_\nu(\mathbb D))$, which is positive and bounded away from zero
uniformly in $z$ and locally uniformly in $a$. Therefore, it suffices to estimate the norms of $(F_\nu)'$ 
for the evaluation of $\Vert \Phi-\Phi_0 \Vert_{{\rm BMOA}}$:
\begin{align*}
\Vert \Phi-\Phi_0 \Vert_{\rm BMOA}&\lesssim
\sup_{x \in \mathbb S,\,r>0} \frac{1}{r}\int_{\Delta(x,r) \cap \mathbb D}(1-|z|^2)|(F_\nu)'(z)|^2 dxdy\\
&\lesssim \sup_{x \in \mathbb S,\,r>0} \int_{\Delta(x,r) \cap \mathbb D}|(F_\nu)'(z)|^2 dxdy \leq \int_{\mathbb D}|(F_\nu)'(z)|^2 dxdy <\infty.
\end{align*}

This estimate shows that $\Vert \widetilde L(\nu,a) \Vert_{{\rm BMOA}}$ is bounded 
by a constant
depending only on $\Vert \log(F_\nu)' \Vert_{\rm BMOA}$, 
$d(a,\partial F_\nu(\mathbb D))$, and the Euclidean area of $F_\nu(\mathbb D)$.
For a given $(\nu_0,a_0) \in \widetilde M_B(\mathbb D^*)$,
these quantities vary within a bounded range 
as $\nu \in M_B(\mathbb D^*)$ and $a \in F_{\nu_0}(\mathbb D^*)$ vary slightly around $(\nu_0,a_0)$.
This implies that $\widetilde L$ is locally bounded.

Under this local boundedness, 
if $\widetilde L$ is G\^{a}teaux holomorphic, then it is holomorphic (see \cite[Theorem 14.9]{Ch}).
Moreover, the G\^{a}teaux holomorphy of $\widetilde L$ follows from the condition that
for each fixed $z \in \mathbb D$,
$\widetilde L(\nu,a)(z)=\log (W_{a} \circ F_{\nu})'(z)$ is 
G\^{a}teaux holomorphic as a complex-valued function (see \cite[Lemma 6.1]{WM-2}). 
This can be verified by the holomorphic dependence of quasiconformal mappings on the Beltrami coefficients (see \cite[Theorem V.5]{Ah}).
Thus, we conclude that $\widetilde L$ is holomorphic on $\widetilde M_B(\mathbb D^*)$.
\end{proof}

\begin{proposition}\label{bundle}
$J:\widetilde{\mathcal T}_B \to \alpha(T_B)$ is a holomorphic split submersion. 
\end{proposition}

\begin{proof}
Let $\Phi \in \widetilde{\mathcal T}_B$ and set $\Psi_0=J(\Phi) \in \alpha(T_B)$. Then,
there exists a
neighborhood $V_{\Psi_0}$ of $\Psi_0$ in $\alpha(T_B)$ 
and a holomorphic map $\sigma:V_{\Psi_0} \to M_B(\mathbb D^*)$ such that
$S \circ \sigma$ is the identity on $V_{\Psi_0}$ by Proposition \ref{weaksubmersion}. 
Let $\Phi_0=L \circ \sigma(\Psi_0)$, which may differ 
from $\Phi$.
Since $\Phi_0$ can be written as $\log (F_{\sigma(\Psi_0)})' \in \widetilde{\mathcal T}_B$, Proposition \ref{affine} implies that
$\Phi=\log (W_a \circ F_{\sigma(\Psi_0)})'$ for some $a \in F_{\sigma(\Psi_0)}(\mathbb D^*)$,
i.e., $\Phi=\widetilde L(\sigma(\Psi_0),a)$.

Fixing this $a$, we define the map
$\widetilde L(\sigma(\cdot),a):V_{\Psi_0} \to \widetilde{\mathcal T}_B$, shrinking $V_{\Psi_0}$ if necessary.
By Lemma \ref{localbound}, this is a holomorphic map on $V_{\Psi_0}$.
Since $J \circ \widetilde L(\sigma(\Psi),a)=\Psi$ for all $\Psi \in V_{\Psi_0}$, the map
$\widetilde L(\sigma(\cdot),a)$ is a local holomorphic right inverse of $J$ passing through the point 
$\Phi=\widetilde L(\sigma(\Psi_0),a)$.
This shows that $J$ is a holomorphic split submersion.
\end{proof}

\begin{remark}
This proposition appears in \cite[Theorem 6.3]{SW}, but it is shown in the weaker form
mentioned in Remark \ref{explanation}.
In contrast to the above result,
it remains unknown whether the pre-Schwarzian derivative map $L:M_B(\mathbb D^*) \to {\rm BMOA}(\mathbb D)$ or
the Schwarzian derivative map $S:M_B(\mathbb D^*) \to A_B(\mathbb D)$ is a holomorphic split submersion onto its image.
\end{remark}

%From here on, we denote by $\widetilde F_\mu$ the quasiconformal extension of $F_\mu$ to $\mathbb D^*$.
%For any $\Psi_0 \in \alpha(T_B)$, the local holomorphic right inverse $\sigma:V_{\Psi_0} \to M_B(\mathbb D^*)$
%as in the above proof is constructed by producing a family of quasiconformal homeomorphisms $\widetilde F_\mu$,  
%explicitly defined by starting with a suitable quasiconformal reflection based on
%a bi-Lipschitz diffeo\-morphism $\widetilde F_\nu$ with $S(\nu)=\Psi_0$,
%and then solving the Schwarzian differential equation for $\Psi \in V_{\Psi_0}$ on $F^\nu(\mathbb D)$.
%See \cite[Section II.4.2]{Le} for the general arguments in the case of the universal Teich\-m\"ul\-ler space.
%The following proposition explains this construction in such a manner that it possesses a certain extra property.

From here on, we denote by $\widetilde F_\mu$ the quasiconformal extension of $F_\mu$ to $\mathbb D^*$.
For any $\Psi_0 \in \alpha(T_B)$, the local holomorphic right inverse $\sigma:V_{\Psi_0} \to M_B(\mathbb D^*)$
as in the proof above is constructed by producing a family of quasiconformal homeomorphisms $\widetilde F_\mu$,  
obtained explicitly by starting with a suitable quasiconformal reflection based on
a bi-Lipschitz diffeo\-morphism $\widetilde F_\nu$ with $S(\nu)=\Psi_0$,  
and then solving the Schwarzian differential equation for $\Psi \in V_{\Psi_0}$ on $F^\nu(\mathbb D)$.
See \cite[Section~II.4.2]{Le} for the general argument in the case of the universal Teich\-m\"ul\-ler space.
The following proposition clarifies this construction in such a way that it enjoys an additional property.

\begin{proposition}\label{real-analytic}
For any $\Psi_0 \in \alpha(T_B)$, there exists a
local holomorphic right inverse $\sigma:V_{\Psi_0} \to M_B(\mathbb D^*)$ of $S$ such that
$\widetilde F_{\sigma(\Psi)}$ is a real-analytic diffeomorphism (=bi-real-analytic homeomorphism) of 
$\mathbb D^*$ for 
any $\Psi \in V_{\Psi_0}$. Moreover, $\widetilde F_{\sigma(\Psi)}(\zeta)$ is jointly real-analytic
for $(\Psi, \zeta) \in V_{\Psi_0} \times \mathbb D^*$.
\end{proposition}

\begin{proof}
For $\Psi_0 \in \alpha(T_B)$, using the argument in \cite[Theorem 9]{M0}, we can show that there exists
$\nu \in M_B(\mathbb D^*)$ such that $S(\nu)=\Psi_0$, and the quasiconformal homeomorphism
$\widetilde F_\nu:\mathbb D^* \to \Omega^*$ with complex dilatation $\nu$
is a real-analytic bi-Lipschitz diffeomorphism with respect to the hyperbolic metrics on $\mathbb D^*$ and
its image domain $\Omega^* \subset \mathbb C$. 
Its conformal extension is $F_{\nu}:\mathbb D \to \Omega={\mathbb C} \setminus \overline{\Omega^*}$. 
Then, a quasiconformal reflection $r:\Omega^* \to \Omega$
with respect to $\partial \Omega=\partial \Omega^*$ is defined by 
\[
r(w)=F_\nu\left((\widetilde F_{\nu})^{-1}(w)^*\right) \quad (w \in \Omega^*),
\]
which is a real-analytic bi-Lipschitz diffeomorphism. Here, $\zeta^*=1/\bar \zeta \in \mathbb D$ for $\zeta \in \mathbb D^*$.

For any $\Psi \in V_{\Psi_0}$, we consider the pushforward $(F_\nu)_*(\Psi)$ by the conformal
homeomorphism $F_\nu:\mathbb D \to \Omega$, and solve the differential equation
$2\psi''(z)+(F_\nu)_*(\Psi)(z)\psi(z)=0$ on $\Omega$.
Let $\psi_1$ and $\psi_2$ be linearly independent solutions, normalized so that $\psi_1\psi_2'-\psi_2\psi_1'=1$.
Then, $S(\psi_1/\psi_2)=(F_\nu)_*(\Psi)$ on $\Omega$, 
and
the quasiconformal homeomorphism $\widetilde F_\mu$ of $\mathbb D^*$ with complex dilatation $\mu=\sigma(\Psi)$ is given by
the composition of $w=\widetilde F_\nu(\zeta)$ with
\begin{equation}\label{generalAW}
\frac{\psi_1(r(w))+(w-r(w))\psi_1'(r(w))}{\psi_2(r(w))+(w-r(w))\psi_2'(r(w))},
\end{equation}
which is a quasiconformal real-analytic diffeomorphism of $w \in \Omega^*$. This fact can be shown using \cite[Lemma 4]{Su}
along with the subsequent comment and remark. In particular, $\widetilde F_\mu=\widetilde F_{\sigma(\Psi)}$ is a real-analytic
diffeomorphism of $\mathbb D^*$.

Concerning the joint real-analyticity of $\widetilde F_{\sigma(\Psi)}(\zeta)$, 
besides relying on a general theory of separate real-analyticity
(see e.g. \cite[Theorem 7.1]{Si}),
we can also obtain it by examining \eqref{generalAW} to see that $\psi_i$ and $\psi'_i$ $(i=1,2)$ are jointly holomorphic
and $r$ is the fixed quasiconformal reflection which is real-analytic on $\zeta$.
%Concerning the joint real-analyticity of $\widetilde F_{\sigma(\Psi)}(\zeta)$, 
%we only have to examine \eqref{generalAW} to see that $w_i$ and $w'_i$ $(i=1,2)$ are jointly holomorphic
%and $r$ is the fixed quasiconformal reflection which is real-analytic on $\zeta$.
\end{proof}

The {\it Bers fiber space} $\widetilde T_B$ over $\alpha(T_B) \subset A_B(\mathbb D)$ is defined as
\[
\widetilde T_B=\{(\Psi,a) \in \alpha(T_B) \times \widehat{\mathbb C} \mid \Psi=S(\nu),
\ a \in \widetilde F_{\nu}(\mathbb D^*),\ \nu \in M_B(\mathbb D^*)\}.
\]
Via the Bers embedding, $\alpha(T_B)$ is identified with
the BMO Teich\-m\"ul\-ler space $T_B$. 
We note that the quasidisk $\widetilde F_{\nu}(\mathbb D^*)$ is determined by $\Psi$
independently of the choice of $\nu \in M_B(\mathbb D^*)$ with $S(\nu)=\Psi$.
We define a map $\lambda: \widetilde T_B \to \widetilde{\mathcal T}_B$ by
\[
\lambda(\Psi,a)=\widetilde L(\nu,a) =\log (W_a \circ F_{\nu})'
\] 
for $S(\nu)=\Psi$. This is well defined and independent of
the choice of $\nu$.

Note that the condition $a \in \widetilde F_{\nu}(\mathbb D^*)$ is equivalent to requiring that
$W_a \circ F_{\nu}$ maps $\mathbb D$ onto a bounded domain in $\mathbb C$, and 
that $a=\infty$ if and only if $W_a$ is an affine transformation of $\mathbb C$.
Hence, by Proposition \ref{affine},
$\lambda$ is bijective.
In fact, $\lambda$ is bijective on each fiber. That is, for each $\Psi \in \alpha(T_B)$ with $S(\nu)=\Psi$, 
$\lambda(\Psi,\cdot)$ maps $\widetilde F_{\nu}(\mathbb D^*)$ bijectively onto 
$J^{-1}(\Psi) \subset \widetilde{\mathcal T}_B$.
Here, $J^{-1}(\Psi)$ is a $1$-dimensional complex submanifold of $\widetilde{\mathcal T}_B$ since 
$J$ is a holomorphic split submersion by Proposition \ref{bundle}.

\begin{lemma}\label{biholo}
The map $\lambda:\widetilde T_B \to \widetilde{\mathcal T}_B$ is a biholomorphic homeomorphism.
\end{lemma}

\begin{proof}
Choose any $\Psi_0 \in \alpha(T_B)$, and take $V_{\Psi_0}$ and $\sigma$ as in the proof of Proposition \ref{bundle}.
The restriction of $\lambda$ to the domain
\[
\widetilde V_{\Psi_0}=\{(\Psi,a) \in V_{\Psi_0} \times \widehat{\mathbb C} \mid 
a \in \widetilde F_{\sigma(\Psi)}(\mathbb D^*)\} \subset \widetilde T_B
\]
is explicitly represented as 
$\lambda_\sigma(\Psi,a)=\widetilde L(\sigma(\Psi),a)$. Then,
we see that $\lambda_\sigma$ is holomorphic on $\widetilde V_{\Psi_0}$ by Lemma \ref{localbound}, and thus $\lambda$ is a holomorphic bijection. 

Moreover, for each fixed $\Psi \in V_{\Psi_0}$, the domain $\widetilde F_{\sigma(\Psi)}(\mathbb D^*)$ of complex dimension $1$
is mapped by $\lambda_\sigma(\Psi,\cdot)$
holomorphically and bijectively onto the complex submanifold $J^{-1}(\Psi) \subset \widetilde{\mathcal T}_B$.
Hence, $\lambda_\sigma(\Psi,\cdot)$ is a biholomorphic homeomorphism.
It follows from this fiber-wise property 
that $\lambda^{-1}$ is holomorphic, and thus $\lambda$ is biholomorphic.
\end{proof}

The structure of the space $\widetilde T_B \cong \widetilde{\mathcal T}_B$ over 
$T_B \cong \alpha(T_B)$ can be described precisely as follows:

\begin{theorem}\label{bundle2}
$\widetilde{\mathcal T}_B$ is a real-analytic disk bundle over $\alpha(T_B)$ with projection $J$.
\end{theorem}

\begin{proof}
By Lemma \ref{biholo}, the map
$\lambda_\sigma(\Psi,a)=\widetilde L(\sigma(\Psi),a)=\log (W_a \circ F_{\sigma(\Psi)})'$ is a biholomorphic homeomorphism
on some neighborhood $\widetilde V_{\Psi_0}$ for any $\Psi_0 \in \alpha(T_B)$.
Using this, we define a disk bundle structure over $\alpha(T_B)$ on $\widetilde{\mathcal T}_B$.
For every $\Psi_0 \in \alpha(T_B)$, we define the bijective map
\[
\ell_\sigma:V_{\Psi_0} \times \mathbb D^* \to J^{-1}(V_{\Psi_0}) \subset \widetilde{\mathcal T}_B
\]
by
$\ell_\sigma(\Psi, \zeta)=\lambda_\sigma(\Psi,\widetilde F_{\sigma(\Psi)}(\zeta))$, which satisfies  
$J \circ \ell_\sigma(\Psi,\zeta)=\Psi$.

Moreover, we choose the special local right inverse $\sigma$ as in Proposition \ref{real-analytic}. Then
$\ell_\sigma$ is a real-analytic diffeomorphism, since $\lambda_\sigma$ is biholomorphic and
$\widetilde F_{\sigma(\Psi)}(\zeta)$ is a jointly real-analytic map into $\widehat{\mathbb C}$
with non-degenerate derivative. 
Hence $\ell_\sigma$ yields a local trivialization of the projection 
$J:\widetilde{\mathcal T}_B \to \alpha(T_B)$, showing that $\widetilde{\mathcal T}_B$ indeed carries the structure
of a real-analytic disk bundle, as claimed.
%Moreover, we choose the special local right inverse $\sigma$ as in Proposition \ref{real-analytic}. Then,
%$\ell_\sigma$ is a real-analytic diffeomorphism since $\lambda_\sigma$ is biholomorphic and
%$\widetilde F_{\sigma(\Psi)}(\zeta)$ is a jointly real-analytic map to $\widehat{\mathbb C}$
%with non-degenerate derivative. 
%Hence, $\ell_\sigma$ provides a local trivialization for the projection 
%$J:\widetilde{\mathcal T}_B \to \alpha(T_B)$, showing that $\widetilde{\mathcal T}_B$ has the structure
%of a real-analytic disk bundle as claimed.
\end{proof}

%We note that since the fiber is the disk $\mathbb D^*$,
%the bundle $\widetilde{\mathcal T}_B \to \alpha(T_B)$ is $C^\infty$-trivial, 
%that is, $\widetilde{\mathcal T}_B$ has a smooth product structure
%$\alpha(T_B) \times \mathbb D^*$.

By restricting the bundle projection $J$ to $\widetilde{\mathcal T}_V \subset {\rm VMOA}(\mathbb D)$,
we can show an analogous result for this sub-bundle with only minor modifications to the proof. We state it as a corollary:

\begin{corollary}\label{bundleVMO}
$\widetilde{\mathcal T}_V$ is a real-analytic disk sub-bundle over $\alpha(T_V)$ with projection $J$.
\end{corollary}

In this case, by \cite[Theorem 6.4]{WM-2},
a global section of the bundle projection $J$ can be obtained via the global real-analytic right inverse $\Sigma$ of
the Schwarzian derivative map $S:M_V(\mathbb D^*) \to \alpha(T_V)$.
Moreover, for each fixed $\Psi \in \alpha(T_V)$, $\widetilde F_{\Sigma(\Psi)}$ is a real-analytic
diffeomorphism of $\mathbb D^*$ (see also \cite[Theorem 4]{M0}), and in fact,  
$\widetilde F_{\Sigma(\Psi)}(\zeta)$ is jointly real-analytic for $(\Psi,\zeta) \in \alpha(T_V) \times \mathbb D^*$.
Replacing the local right inverse $\sigma$ in the proof of Theorem \ref{bundle2} 
with this global $\Sigma$,
we define a real-analytic diffeomorphism
\[
\ell_\Sigma:\alpha(T_V) \times \mathbb D^* \to \widetilde{\mathcal T}_V
\]
by
$\ell_\Sigma(\Psi, \zeta)
=\widetilde L(\Sigma(\Psi),\widetilde F_{\Sigma(\Psi)}(\zeta))$. 
Then, in
the real-analytic category, the total space $\widetilde{\mathcal T}_V$ has a product structure,
and the bundle becomes trivial. 
%We note that $T_V$ is real-analytically contractible by \cite[Corollary 1.3]{WM-2},
%from which the triviality of the bundle follows. We claim here that the map $\ell_\Sigma$ gives an explicit
%real-analytic equivalence.

\begin{theorem}\label{product}
Each $\zeta \in \mathbb D^*$ defines a global real-analytic section 
\[
\ell_\Sigma(\cdot, \zeta):\alpha(T_V) \to \widetilde{\mathcal T}_V
\]
for the bundle projection $J:\widetilde{\mathcal T}_V \to \alpha(T_V)$.
Moreover, the total space $\widetilde{\mathcal T}_V$ is real-analytically equivalent to 
$\alpha(T_V) \times \mathbb D^*$
under $\ell_\Sigma$, satisfying $J \circ \ell_\Sigma(\Psi,\zeta)=\Psi$.
\end{theorem}

We note that $T_B \cong \alpha(T_B)$ is also real-analytically contractible because 
$T_B$ is real-analytically equivalent to a certain convex domain of the Banach space of real-valued BMO functions by
\cite[Corollary 1.2]{WM-1}. However, this fact does not help us to find whether or not
the trivial disk bundle $\widetilde{\mathcal T}_B$ over $\alpha(T_B)$ is real-analytically trivial.
%We only have that $\widetilde{\mathcal T}_B$ is topologically trivial.
%Thus, for the same reason, the triviality of the bundle over $T_B$ follows.
%We claim here that the map $\ell_\Sigma$ gives an explicit
%real-analytic equivalence.

%\begin{theorem}\label{product2}
%The real-analytic disk bundle $J:\widetilde{\mathcal T}_B \to \alpha(T_B)$ is trivial. In particular, there exists
%a global real-analytic section $\eta:\alpha(T_B) \to \widetilde{\mathcal T}_B$ such that 
%$J \circ \eta={\rm id}_{\alpha(T_B)}$.
%\end{theorem}

%\begin{remark}
%All the results in this section are also known to hold for the universal Teich\-m\"ul\-ler space $T$
%with $M(\mathbb D^*)$ and $B(\mathbb D)$ (see \cite[Appendix]{Teo}), or can be verified more easily in that setting.
%Analogous arguments and results can be applied also to other Teich\-m\"ul\-ler spaces such as those for
%circle diffeomorphisms of H\"older or Zygmund continuous derivatives and integrable Beltrami coefficients
%giving the Weil--Petersson metric.
%\end{remark}

\begin{remark}
All the results in this section are also known to hold for the universal Teich\-m\"ul\-ler space $T$,
with $M(\mathbb D^*)$ and $B(\mathbb D)$ (see \cite[Appendix]{Teo}), and can in fact be verified more easily in that setting.
Analogous arguments and results also apply to other Teich\-m\"ul\-ler spaces, such as those arising from
circle diffeomorphisms with H\"older or Zygmund continuous derivatives, and from integrable Beltrami coefficients
giving the Weil--Petersson metric.
\end{remark}

\end{document}